\documentclass[10pt,a4paper,oneside]{amsart}

\usepackage{latexsym, amsthm}
\usepackage{amsfonts, amsmath, amssymb}
\usepackage{euscript,mathrsfs}
\usepackage{url}
\usepackage{array}
\usepackage[a4paper]{geometry}
\usepackage{graphics}
\usepackage[usenames]{color}
\usepackage[all]{xy}
\usepackage{graphicx}
\usepackage{boxedminipage}

\newtheorem{theorem}{Theorem}[section]

\newtheorem{conjecture}{Conjecture}

\newtheorem{lemma}[theorem]{Lemma}

\newtheorem{proposition}[theorem]{Proposition}

\newcommand{\fq}{\mathbb{F}_q}

\newcommand{\N}{\mathbb{N}}
\newcommand{\Z}{\mathbb{Z}}

\newcommand{\Pp}{\mathbb{P}}
\newcommand{\Phat}{\mathbb{\hat{\mathbb{P}}}}

\newcommand{\eim}{\mathbf{e}_i^m}

\def\Fq{{\mathbb F}_q}
\def\AA{{\mathbb A}}
\def\FF{{\mathbb F}}
\def\PP{{\mathbb P}}
\newcommand{\V}{\mathsf{V}}
\newcommand{\Ze}{\mathsf{Z}}

\newcommand{\X}{\mathcal{X}}
\newcommand{\Hy}{\mathcal{H}}

\begin{document}

\title[Zeros of Homogeneous Polynomials over Finite Fields]{Maximum Number of Common Zeros of  Homogeneous Polynomials over Finite Fields}
\author{Peter Beelen}
\address{Department of Applied Mathematics and Computer Science, \newline \indent
Technical University of Denmark, DK 2800, Kgs. Lyngby, Denmark}
\email{pabe@dtu.dk}

\author{Mrinmoy Datta}
\address{Department of Applied Mathematics and Computer Science, \newline \indent
Technical University of Denmark, DK 2800, Kgs. Lyngby, Denmark}
\email{mrinmoy.dat@gmail.com}

\author{Sudhir R. Ghorpade}
\address{Department of Mathematics,
Indian Institute of Technology Bombay,\newline \indent
Powai, Mumbai 400076, India.}
\email{srg@math.iitb.ac.in}


\date{}

\begin{abstract}
About two decades ago, Tsfasman and Boguslavsky conjectured a formula for the
maximum number of common zeros that $r$  linearly independent homogeneous polynomials
of degree $d$ in $m+1$ variables with coefficients in a finite field with $q$ elements can have in the corresponding $m$-dimensional projective space. 
Recently, it has been shown by Datta and Ghorpade that this conjecture is valid if $r$ is at most $m+1$ and can be invalid otherwise. Moreover a new conjecture was proposed for many values of $r$ beyond $m+1$. In this paper, we prove that this new conjecture holds true for several values of $r$. In particular, this settles the new conjecture completely when $d=3$. Our result also includes the positive result of Datta and Ghorpade as a special case. Further, we determine the maximum number of zeros in certain cases not covered by the earlier conjectures and results, namely, the case of $d=q-1$ and of $d=q$. All these results are directly applicable to the determination of the maximum number of points on sections of Veronese varieties by linear subvarieties of a fixed dimension, and also the determination of generalized Hamming weights of projective Reed-Muller codes.
\end{abstract}

\maketitle

\section{Introduction}
\label{sec:in}

Let $d,m$ be positive integers and let $\Fq$ denote the finite field with $q$ elements.  Let  us denote by $S$ the ring $\Fq[X_0,X_1, \dots , X_m]$ 
of polynomials in $m+1$ variables with coefficients in $\Fq$ and by $S_d$ 
its $d$th graded component, i.e., let $S_d$ be the space of all homogeneous polynomials in $S$ of degree $d$ (including the zero polynomial). Given any homogeneous polynomials $F_1, \dots , F_r \in S$, let
$ \V(F_1, \dots , F_r)$ denote the corresponding projective algebraic variety over $\Fq$, i.e., the set of all
$\Fq$-rational common zeros of $F_1, \dots , F_r$ in the $m$-dimensional projective space $\PP^m$.
Now fix a positive integer $r \le \dim_{\Fq}S_d = \binom{m+d}{d}$. We are
primarily interested in determining
\begin{equation}
\label{erdm}
e_r(d,m): = \max\left\{ \left| \V(F_1, \dots , F_r) \right| : F_1, \dots , F_r \in S_d \text{ linearly independent}\right\} .
\end{equation}
The first nontrivial case is $r=1$ and here
it was conjectured by Tsfasman in the late 1980's that
\begin{equation}
\label{TConj}
e_1(d,m) = dq^{m-1} + p_{m-2} \quad \text{whenever } d\le q,
\end{equation}
where for any integer $k$,
\begin{equation*}
p_k: = \begin{cases} |  \PP^k(\Fq)| = q^k + q^{k-1} + \dots + q + 1 
& \text{ if } k \ge 0, \\
0 &  \text{ if } k<0. \end{cases}
\end{equation*}
The conjecture was proved in the affirmative by Serre \cite{Se} and, independently, by S{\o}rensen  \cite{Soe} in 1991. Later in 1997,  Boguslavsky \cite{Bog} showed that
$$
e_2(d,m)=(d-1)q^{m-1}+ q^{m-2} + p_{m-2}  \quad \text{whenever } 1< d < q-1.
$$
In the same paper, Boguslavsky \cite{Bog} made several conjectures, ascribing some of them to Tsfasman. Surmising from the conjectural statements and results in \cite{Bog}, one arrives at the
Tsfasman-Boguslavsky Conjecture (TBC), which states that
\begin{equation*}
e_r(d,m): = \displaystyle{   p_{m-2j} + \sum_{i=j}^m} \nu_i (p_{m-i} - p_{m-i-j}) \quad \text{whenever } 1\le d < q-1,
\end{equation*}
where $(\nu_1, \dots, \nu_{m+1})$ is the $r$th element in descending lexicographic order among
$(m+1)$-tuples $(\alpha_1, \dots , \alpha_{m+1})$ of nonnegative integers satisfying
$\alpha_1+ \cdots + \alpha_{m+1} = d$, and 
$j := \min\{i : \nu_i \ne 0\}$.

 The conjectural formula above for $e_r(d.m)$ was motivated by the computations of Boguslavsky \cite[Lem.~4]{Bog} for the number of $\Fq$-rational points of the so-called linear $(r,m,d)$-configurations, and a conjecture of Tsfasman \cite[Conj.~1]{Bog}. For details about these, 
  see \cite[\S\,2]{Bog} and \cite[Rem.~3.6]{DG1}.
The TBC remained open for a considerably long time. However, two important developments took place shortly after Boguslavsky's paper was published. First, working on a seemingly unrelated question (and unaware of the TBC), Zanella \cite{Z} determined $e_r(2,m)$ completely. Second, Heijnen and Pellikaan \cite{HP}, found exact formulae for the affine analogue of \eqref{erdm}, namely,
\begin{equation*}
e_r^{\AA}(d,m): = \max\left\{ \left| \Ze(f_1, \dots , f_r) \right| :  f_1, \dots , f_r \in T_{\le d} \; \text{ linearly independent}\right\} ,
\end{equation*}
where $T$ denotes the polynomial ring $\Fq[x_1, \dots , x_m] $ in $m$ variables over $\Fq$ and $T_{\le d}\; $
the set of polynomials in $T$ of degree $\le d$, and for any $f_1, \dots , f_r\in T$,  $\Ze(f_1, \dots , f_r)$ denotes the set of all $\Fq$-rational common zeros of $f_1, \dots , f_r$ in the $m$-dimensional affine space $\AA^m$. 
The result of Heijnen-Pellikaan can be stated as follows.
\begin{equation}
\label{Hrdm}
e_r^{\AA}(d,m) = H_r(d,m):= \sum_{i=1}^m \beta_i q^{m-i} \quad \text{whenever } 1\le d < q, \; m\ge 1, \text{ and } 1\le r \le \binom{m+d}{d},
\end{equation}
where $(\beta_1, \dots, \beta_{m})$ is the $r$th element in descending lexicographic order among  all
$m$-tuples $(\gamma_1, \dots , \gamma_{m})$ of nonnegative integers satisfying
$\gamma_1+ \cdots + \gamma_{m} \le d$.

Recently, it was shown 
in \cite{DG1} that the TBC is false, in general, by showing that $e_r(d,m)$ can be strictly smaller than the conjectured quantity if $r>m+1$. Further, in \cite{DG} 
it was shown that the TBC holds in the affirmative if $r\le m+1$; this gives
\begin{equation}
\label{DGrdm}
e_r(d,m)=(d-1)q^{m-1}+ \lfloor q^{m-r} \rfloor + p_{m-2}  \quad \text{if } 1< d < q-1 \text{ and } r \le m+1.
\end{equation}
While this settles in a way the Tsfasman-Boguslavsky Conjecture, there still remains the question of determining $e_r(d,m)$ in all the remaining cases. 
In fact, besides \eqref{TConj}, \eqref{DGrdm}, and the result of Zanella for $e_r(2,m)$ mentioned earlier (see Theorem \ref{thm:Z}), the only other known results about $e_r(d,m)$ are the following. First, it is easy to determine $e_r(d,m)$ for the initial values of $d$ or $m$ for all permissible $r$, that is, for $1\le  r \le  \binom{m+d}{d}$. More precisely, we have
\begin{equation}
\label{ermone}
e_r(1,m) = p_{m-r} \text{ for } 1\le r \le m+1.\quad \text{and} \quad e_r(d,1) =  d-r+1 \quad \text{for } 1\le r \le d+1 \text{ and } d\le q; 
\end{equation}
see, for instance,  \cite[\S\,2.1]{DG}. It is not difficult to determine $e_r(d,m)$ for some terminal values of $r$:
\begin{equation}
\label{elastr}
e_r(d,m) = \textstyle{\binom{m+d}{d}} - r \quad \text{for } \,  \binom{m+d}{d} - d \le r \le  \binom{m+d}{d} \, \text{ and } \, d< q-1.
\end{equation}
A proof can be found in \cite[Thm. 4.7]{DG2}.
At any rate, the results obtained thus far do not yield the exact values of
\begin{itemize}
\item
$e_r(d,m)$ whenever $m+1 < r < \binom{m+d}{d} - d$ and $2< d < q-1$
\item
$e_r(d,m)$ whenever $1 < r \le \binom{m+d}{d}$ and $d \ge q-1$
\end{itemize}
Note also that the case $d\ge q+1$ is trivial for many values of $r$ (see \cite[Rem. 6.2]{DG} for more details).
But the cases $d=q-1$ and $d=q$ were unresolved for most values or $r$ and $m$,
and it is conceivable that the TBC may even be valid in some of them, at least when $r\le m+1$. 
For going beyond $r=m+1$, 
a conjecture that ameliorates the TBC was made in \cite{DG} for many (but not all) values of $r$ and for values of $d$ up to and including $q-1$. The conjecture simply states that
\begin{equation}
\label{DGdm}
e_r(d,m) = H_r(d-1,m) + p_{m-1} \quad \text{if } 1< d \le q-1 \text{ and } r \le {{m+d-1}\choose{d-1}}.
\end{equation}
where $H_r(d-1,m)$ is as in \eqref{Hrdm} except with $d$ replaced by $d-1$.

We can now describe the contents of this paper. Our main result (Theorem \ref{thm:main}) is an affirmative solution of the new conjecture \eqref{DGdm} when $d>2$ and $r \le {{m+2}\choose{2}}$. In particular, this completely proves the conjectural formula \eqref{DGdm} when $d=3$. Furthermore, while our methods are partly inspired by those in \cite{DG}, the results of \cite{DG} are not used directly. As such our results yield \eqref{DGrdm} as a corollary. In fact, we do a little better, since the case $d=q-1$ is also covered, and moreover, the proof is somewhat simpler. Our second main result (Theorems  \ref{thm:erdm3}
and \ref{pro:erdm4}) is the determination of $e_r(d,m)$ in the case $d=q$ and $1\le r \le m+1$. The result matches with the answer predicted by the TBC as well as \eqref{DGrdm} and \eqref{DGdm} when $r=1$ and $r=m+1$, but not otherwise.

The key ingredients in our proofs  are as follows. We make
use of the nontrivial results of Heijnen and Pellikaan \cite{HP} as well as Zanella \cite{Z}. 
 In addition, we utilize an inequality of Serre/S{\o}rensen \cite{Se, Soe}, a variant of B\'{e}zout's theorem by Lachaud and Rolland \cite{LR},
 a simple lemma given by Zanella \cite{Z} (see also \cite[Rem. 2.3]{DG1}),
 and an inequality of Homma and Kim \cite{HK} about the maximum number of points on a hypersurface without an $\Fq$-linear component. Another important ingredient in our proof is the use of a quantity that we call the $t$-invariant associated to a linear space of homogeneous polynomials of the same degree. This notion can be traced back to the proof of \cite[Thm. 5.1]{DG2} in a special case, but here it is used more systematically.

We remark here that the determination of $e_r(d,m)$ is equivalent to the determination of the maximum number of $\Fq$-rational points on linear sections ${\mathscr{V}_{m,d}}\cap L$ of the Veronese variety ${\mathscr{V}_{m,d}}$ corresponding to the $d$-uple embedding of $\PP^m$ in $\PP^{M-1}$, where $M=  \binom{m+d}{d}$ and where $L$ varies over linear subvarieties of  $\PP^{M-1}$ of codimension $r$. Moreover, finding
$e_r(d,m)$ is essentially the same as finding the $r$th generalized Hamming weight of the projective Reed-Muller code ${\mathrm{PRM}}_q(d,m)$ of order $d$ and length $p_m$. Also, results on the determination of $e_r(d,m)$ complement the recent result of Couvreur \cite{C} on the number of points of projective varieties of given dimensions and degrees of its irreducible components.
These connections are explained in \cite{DG1,DG2}, and one may refer to them for more details on these aspects.
\section{Preliminaries}
\label{sec:Prelim}

Fix for the remainder of this paper a prime power $q$ and positive integers $d,m,r$. In subsequent sections and subsections, some further assumptions on $d$ or $m$ or $r$ will be made, depending on the context. For the convenience of the reader, the basic underlying assumptions, if any, will be specified in the ``context'' mentioned at the beginning of the section or subsection. We will denote by $\N$ the set of all nonnegative integers and by $\N^m$ the set of $m$-tuples of nonnegative integers. 
We will continue to use the notations introduced in the previous section.
In particular, given any subset $W$ of $S:=\fq[X_0, X_1,\dots,X_m]$, we denote by $\V(W)$ the set of $\Fq$-rational points of the corresponding projective variety in  $\PP^m$, i.e., $\V(W):=\{P\in \PP^m(\Fq): F(P)=0 \text{ for all } F\in W\}$.  If $W=\{F_1, \dots , F_r\}$ or if $W$ is a $\Fq$-linear subspace of $S$ spanned by $F_1, \dots , F_r$, then we
may write $\V(F_1, \dots , F_r)$ for $\V(W)$.
Likewise, given any subset $U$ of $T:=\fq[X_1,\dots,X_m]$, we shall denote by $\Ze(U)$ the
set $\{P\in \AA^m(\Fq): f(P)=0 \text{ for all } f\in U\}$. If $U=\{f_1, \dots , f_r\}$ or if $U$ is a 
subspace of $T$ spanned by $f_1, \dots , f_r$, then we may write $\Ze(f_1,\dots,f_r)$ for $\Ze(U)$. Note that we use the word \emph{algebraic variety} as synonymous with \emph{algebraic set}, i.e., a variety need not be irreducible. When we speak of geometric attributes such as dimension or degree of an (affine or projective) algebraic variety such as $\V(W)$ or $\Ze(U)$, it will always be understood that it is the same as the dimension or degree of the corresponding variety over an algebraic closure $\overline{\FF}_q$ of $\Fq$.
%

\subsection{Projective Hypersurfaces and Affine Varieties}
\label{subsec:2.1}
We 
recall here several results from the literature that we will need later on. Let us begin with the
result of Serre \cite{Se} and S{\o}rensen \cite{Soe} (see also \cite{DG1}) that was mentioned in the Introduction.

\begin{theorem}
\label{thm:SerreSorr}
Let $F$ be any nonzero homogeneous polynomial in $S$ of degree $d$. Then
$$
|\V(F)| \le dq^{m-1} + p_{m-2}.
$$
Moreover $e_1(d,m)=dq^{m-1}+p_{m-2}$ whenever $d \le q$.
\end{theorem}

Next, we recall a variant of B\'{e}zout's Theorem given by Lachaud and Rolland \cite[Cor 2.2]{LR}.
It should be noted that since $S$ as well as $T$ are unique factorization domains, 
a gcd (= greatest common divisor) of any finite collection $F_1, \dots , F_r$ of polynomials in either of these rings 
exists and is unique up to multiplication by a nonzero scalar, and it may be denoted by
$\gcd(F_1, \dots, F_r)$. Note also that in case  $F_1, \dots , F_r$ are homogeneous, then so is their gcd.

\begin{theorem}
\label{thm:AffineLachaud}
Let $f_1,\dots,f_r \in T$ be nonzero polynomials such that $\Ze(f_1, \dots , f_r)$ is an affine algebraic variety of dimension $s$. Then 
$$
|\Ze(f_1, \dots , f_r)| \le \deg(f_1)\cdots \deg(f_r) q^{s}.
$$
In particular we have $$|\Ze(f_1)| \le \deg (f_1) q^{m-1}$$ and
$$
|\Ze(f_1,f_2)| \le \deg (f_1) \deg (f_2) q^{m-2}, \quad \text{provided} \quad \gcd(f_1,f_2)=1.
$$
\end{theorem}

\begin{proof}
The first assertion is \cite[Cor 2.2]{LR}. The next two are immediate consequences because if $f_1\in T$ is 
nonconstant, then the hypersurface $\Ze (f_1)$ has codimension $1$ in $\AA^m$, whereas if $f_1, f_2\in T$ are coprime of positive degrees, then arguing as in the proof of \cite[Lem. 2.2]{DG}, we 
see that the codimension of $\Ze(f_1, f_2)$ is $2$. The case when $\deg(f_i)=0$ for some $i=1,2$, is trivial.
\end{proof}

Let us deduce a refinement of the last result, which will be useful to us later.

\begin{lemma}
\label{lem:AffineLachaudrefined}
Assume that $r \ge 2$.
Let $f_1, \dots, f_r \in T_{\le d}$ be  linearly independent polynomials such that $\gcd(f_1,\dots,f_r)=1$.
If $\deg(f_1) \le d-1$, then
$$
|\Ze(f_1,\dots,f_r) | \le (d-1)d q^{m-2}.$$
If, in addition, 
$\deg (f_2) \le d-1$, then
$$
|\Ze(f_1,\dots,f_r)| \le (d-1)^2 q^{m-2}.$$
\end{lemma}

\begin{proof}
For $r=2$ this follows directly from Theorem \ref{thm:AffineLachaud}. Therefore we assume $r>2$ from now on.
To estimate $|\Ze(f_1,\dots,f_r)|$ we proceed as follows: Let $p$ be an irreducible factor of $f_1$. Since we assume that $\gcd(f_1,\dots,f_r)=1$, there exists $i\ge 2$ such that $\gcd(p,f_i)=1$. Using Theorem \ref{thm:AffineLachaud}, we see that $|\Ze(p,f_2,\dots,f_r)| \le |\Ze(p,f_i)| \le \deg(p) d q^{m-2}.$ On the other hand if $f_1=p_1\cdots p_k$ for some irreducible,  but not necessarily distinct,
$p_1, \dots , p_k\in T$, then $|\Ze(f_1,\dots,f_r)|\le \sum_j |\Ze(p_j,f_2,\dots,f_r)|.$ Combining these two estimates, we find that
$$
|\Ze(f_1,\dots,f_r)| \le \deg(f_1)d q^{m-2} \le (d-1)d q^{m-2}.
$$

Now suppose  $\deg (f_1) \le d-1$ and $\deg (f_2) \le d-1$. Here, we need a more refined analysis. Let $g=\gcd(f_1,f_2)$ and write $b=\deg(g)$ and $f_1=gf_1'$, $f_2=gf_2'$. Since $f_1$ and $f_2$ are linearly independent, $f_1'$ and $f_2'$ are nonconstant polynomials.
Hence $b<d-1$. 
It is clear that
$$|\Ze(f_1,\dots,f_r)|\le |\Ze(f_1',f_2')|+|\Ze(g,f_3,\dots,f_r)|.$$
By Theorem \ref{thm:AffineLachaud}, we find that $|\Ze(f_1',f_2')| \le (d-b-1)^2q^{m-2}$. To estimate $|\Ze(g,f_3,\dots,f_r)|$ we proceed on similar lines as before and obtain that
$$|\Ze(g,f_3,\dots,f_r)| \le b d q^{m-2}.$$
Hence we see that
$$|\Ze(f_1,f_2,f_3,\dots,f_r)| \le (d-b-1)^2q^{m-2}+bdq^{m-2}=\left( (d-1)^2+b(b-d+2) \right) q^{m-2}.$$
Since $0 \le b \le d-2$, the maximal value of $b(b-d+2)$ is attained for $b=0$ (or $b=d-2$). The conclusion of the lemma now follows in this case as well.
\end{proof}

We will also need the following result due to Homma and Kim \cite[Thm.1.2]{HK}:

\begin{theorem}\label{thm:HK}
Let $\X \subset \Pp^m(\overline{\FF}_q)$ be a hypersurface of degree $d$ defined over $\mathbb{F}_q$ without an $\mathbb{F}_q$-linear component, and let $\X(\Fq)$ denote the set of its $\Fq$-rational points. Then
$$
|\X (\Fq)| \le (d-1)q^{m-1}+dq^{m-2}+p_{m-3}.
$$
\end{theorem}

The following lemma will play an important role later and it appears, for example, in \cite[Lem. 3.3]{Z}.
See also \cite[Lem. 2.1 and Rem. 2.3]{DG1}. We outline a 
proof for the sake of completeness.

\begin{lemma}\label{lem:Zanella}
Let $\X\subseteq \mathbb{P}^m (\Fq)$ be any 
subset. 
Define
$$a:=\max_{\Hy}|\X \cap \Hy|,$$ where $\Hy$ ranges over all hyperplanes in $\Pp^m$ defined over $\Fq$.
Then
$$
|\X| \le aq +1 \quad \text{and if $\X \ne \PP^m (\Fq)$, then} \quad |\X| \le aq.
$$
\end{lemma}

\begin{proof}
Let $\Phat^m (\Fq)$ denotes the set of hyperplanes in $\Pp^m$ defined over $\fq$.
Counting the incidence set
$
\{ (P, H) \in \X \times \Phat^m (\Fq) : P \in H \}
$
in two ways using the first and the second projections, we obtain $|\X| p_{m-1} \le a p_m$.
This gives $|\X| \le aq + (a/p_{m-1})\le aq+1$, since $a \le p_{m-1}$.
Further, if $a< p_{m-1}$, then  $|\X| \le aq$,
since $|\X|$ is an integer, whereas if $a=p_{m-1}$ and $|\X| = a q+ 1 = p_m$, then we must have $\X=\PP^m (\Fq)$. 
This completes the proof.
\end{proof}
%

We have already alluded to an important result of Heijnen and Pellikaan \cite{HP}. We end this subsection by recording its statement
essentially as in \cite[Thm. 5.10]{HP}, and then outline how the version stated in the Introduction can be deduced.

\begin{theorem} 
\label{HP}
Assume that $1 \le d< q$ and $r\le \binom{m+d}{d}$.
Then
\begin{equation}
\label{HrdmO}
e_r^{\AA}(d,m) = 
q^m - \Big({1+\sum_{j= 1}^{m}} \alpha_{j} q ^{m-j} \Big),
\end{equation}
where $(\alpha_1, \dots , \alpha_m)$ is the $r^{\rm th}$ tuple in ascending lexicographic order among
$m$-tuples $(\lambda_1, \dots , \lambda_m)$ with coordinates from $\{0,1, \dots , q-1\}$ satisfying $\lambda_1+ \dots + \lambda_m \ge m(q-1)-d$,
\end{theorem}

%
To see the equivalence with \eqref{Hrdm}, let us rewrite the expression on the right in \eqref{HrdmO} as
$$
\sum_{j=1}^m (q^{m-j+1} - q^{m - j} - \alpha_{j}q^{m-j} )
=\sum_{j=1}^m \beta_j q^{m-j}, \quad \text{where }
\beta_j:=q-1-\alpha_j \text{ for }j=1, \dots, m.
$$
Note that $(\beta_1,\dots,\beta_m)$ is precisely the $r$th tuple in descending lexicographic order among all $m$-tuples $\gamma = (\gamma_1, \dots , \gamma_m)$ with coordinates in
$\{0,1,\dots,q-1\}$ satisfying 
$\gamma_1 + \cdots +  \gamma_m \le d$.
Moreover, if $d< q$, then the last condition implies 
$\gamma_j \le q-1$ for $j=1, \dots , m$. So if we take 
\begin{equation}\label{eq:Sigmadm}
\Sigma (d,m) := \left\{\gamma = (\gamma_1, \dots , \gamma_m)\in \N^m : \gamma_1 + \cdots +  \gamma_m \le d \right\}
\end{equation}
and $\beta$ the $r$th element of $\Sigma (d,m)$ in descending lexicographic order,
then 
\eqref{HrdmO} implies \eqref{Hrdm}.

\subsection{Combinatorics of $H_r(d,m)$}
As mentioned in the Introduction, we are mainly interested in this paper in conjectural equality \eqref{DGdm} and it is therefore important to understand $H_r(d,m)$ a little better. Let us begin by recalling the definition:
$$
H_r(d,m): = \sum_{j=1}^m \beta_j q^{m-j}, \quad \text{for } m\ge 1, \ 1\le d < q \text{ and } 1\le r \le \textstyle{\binom{m+d}{d}},
$$
where  $\beta$ the $r$th element of $\Sigma (d,m)$ in descending lexicographic order and where $\Sigma (d,m)$ is as in \eqref{eq:Sigmadm}. We shall now proceed to establish several elementary properties of $H_r(d,m)$. These might seem disparate at first, but they will turn out to be useful in later sections.

Observe that if
$\lambda_1, \dots , \lambda_m$ are integers, not all zero, 
with $|\lambda_j|\le q-1$ for $j=1, \dots , m$, then
the sum
$\sum_{j=1}^m \lambda_jq^{m-j}$ 
has the same sign as that of the first nonzero integer among  $\lambda_1, \dots , \lambda_m$.
Now if $d<q$ and if $\gamma, \gamma'\in \Sigma (d,m)$, then 
using the above observation for $\lambda= \gamma - \gamma'$, we see that
$$
\gamma <_{\rm lex} \gamma' \Longleftrightarrow \sum_{j=1}^m\gamma_jq^{m-j} \, < \, \sum_{j=1}^m \gamma'_jq^{m-j}.
$$
This 
implies the strict monotonicity of $H_r(d,m)$ in the parameter $r$:
\begin{equation}
\label{Monotone}
H_1(d,m) > H_2(d,m) >  \dots > H_{ \binom{m+d}{d}}(d,m).
\end{equation}
We will now try to determine $H_r(d.m)$ explicitly for ``small'' values of $r$. 
For $1\le i \le m+1$, let 
$\eim$ be the $m$-tuple with $1$ in $i$th place and $0$ elsewhere; when $i=m+1$, this is 
the zero-tuple. 
Clearly, the first $m+1$ elements of $\Sigma(d,m)$ are 
$(d-1)\mathbf{e}_1^m + \mathbf{e}_{r}^m$  for $r=1, \dots , m+1$. Consequently,
\begin{equation}
\label{Hvsmallr}
H_r(d,m) = (d-1) q^{m-1} + \lfloor q^{m-r} \rfloor \quad \text{for } 1\le r\le m+1 \text{ and } 1\le d < q.
\end{equation}
In particular, if $d=1$, then we have the simple expression $\lfloor q^{m-r} \rfloor$ for $H_r(d,m)$ for all permissible values of $r$. 
Now suppose $2\le d < q$. Then the  first $\binom{m+2}{2}$ elements can be described in
blocks of $(m+1), m , (m-1), \dots , 2, 1$ as follows
\begin{eqnarray*}
&& (d-2) \mathbf{e}_1^m +  \mathbf{e}_1^m +  \mathbf{e}_j^m \quad \text{for} \quad j=1, \dots , m+1, \\
&& (d-2) \mathbf{e}_1^m +  \mathbf{e}_2^m +  \mathbf{e}_j^m \quad \text{for} \quad j=2, \dots , m+1, \\
&& (d-2) \mathbf{e}_1^m +  \mathbf{e}_3^m +  \mathbf{e}_j^m \quad \text{for} \quad j=3, \dots , m+1, \\
&& \quad \vdots  \\
&& (d-2) \mathbf{e}_1^m +  \mathbf{e}_m^m +  \mathbf{e}_j^m \quad \text{for} \quad j=m , m+1, \\
&& (d-2) \mathbf{e}_1^m +    \mathbf{e}_{m+1}^m +    \mathbf{e}_{m+1}^m = (d-2) \mathbf{e}_1^m .
\end{eqnarray*}
Put another way, for $r\le \binom{m+2}{2}$, the $r$th element of $\Sigma (d,m)$
is of the form
\begin{equation}
\label{rthtuple}
(d-2) \mathbf{e}_1^m +  \mathbf{e}_i^m +  \mathbf{e}_j^m \quad \text{for unique $i,j \in \Z$ with } 1\le i \le j \le m+1.
\end{equation}
An easy calculation shows that these unique $i,j$ are related to $r\le \binom{m+2}{2}$ by: 
\begin{equation}
\label{rij}
r = (i-1)m - \textstyle{\binom{i-1}{2} }+ j \quad \text{and} \quad 1\le i \le j \le m+1.
\end{equation}
The conditions \eqref{rij} determine $i,j$ uniquely from a  given $r\le \binom{m+2}{2}$.
  From 
  \eqref{rthtuple}, we  see that 
\begin{equation}
\label{Hsmallr}
H_r(d,m) = (d-2) q^{m-1} + \lfloor q^{m-i} \rfloor + \lfloor q^{m-j} \rfloor \quad \text{for } r\le \textstyle{\binom{m+2}{2}} \text{ with $i,j$ as in \eqref{rij}}.
\end{equation}
Notice that in the above setting $i=1$ if and only if $r\le m+1$ and in this case \eqref{Hsmallr}
simplifies to \eqref{Hvsmallr}, at least when $d\ge 2$.
As an additional illustration of \eqref{Hsmallr}, we may also note that
\begin{equation}
\label{Hmplustwo}
H_{m+2}(d,m) = (d-2) q^{m-1} + 2 \lfloor q^{m-2} \rfloor  \quad \text{for } 
2\le d < q.
\end{equation}
Having observed that $H_r(d,m)$ is strictly monotonic in the parameter $r$, we will examine in the next two results the behavior of $H_r(d,m)$ as a function of the parameter $d$ or the parameter $m$.

\begin{proposition}
\label{Hrvaryingd}
Assume that $1< d < q$ and let $c$ be an integer with $0 < c < d-1$. Then
 \begin{equation}
\label{Hrminust}
H_r(d,m) = c q^{m-1} + H_{r}(d-c, m) \quad \text{for } 1< r \le  \textstyle{\binom{m+2}{2}}.
\end{equation}
In particular, $H_r(d-1,m) < H_r(d,m)$ whenever $2< d < q$ and $1< r \le  \textstyle{\binom{m+2}{2}}$.
\end{proposition}

\begin{proof}
Fix $r$ with $1< r \le  \textstyle{\binom{m+2}{2}}$, and let $i,j$ be as in \eqref{rij}. Then $j\ge 2$. Also $1< d-c < q$.  Thus by \eqref{Hsmallr}, we see that $H_r(d,m) - H_r(d-c,m) = cq^{m-1}$. This implies the desired result.
\end{proof}

\begin{proposition}
\label{pro:Hrdm}
Assume that $1\le d < q$ and $m>1$. 
Then
\begin{enumerate}
\item[(i)] $qH_{r}(d,m-1) \le H_r(d,m)$ whenever $1\le r \le  \binom{m+d-1}{d}$.
\item[(ii)] $qH_{r-1}(d,m-1) \le H_r(d,m)$ whenever $m+1 \le r \le  \binom{m+d-1}{d}$.
\end{enumerate}
\end{proposition}

\begin{proof}
Consider 
$\phi: \Sigma(d,m-1)\to \Sigma (d,m)$ defined by $\phi(\gamma_1, \dots , \gamma_{m-1}) = (\gamma_1, \dots , \gamma_{m-1},0)$. It is clear that $\phi$ preserves lexicographic order and that it maps the first $m-1$ elements of $\Sigma(d,m-1)$ to the first $m-1$ elements of $\Sigma(d,m)$. Thus if $\beta$ is the $r$th element of $\Sigma(d,m-1)$, then $\phi(\beta)$ is the $s$th element of $\Sigma(d,m)$ for some
$s \ge r$. Hence  in view of \eqref{Monotone}, we find, 
for $r \le  \binom{m+d-1}{d}$,
\begin{equation}
\label{qHr}
qH_{r}(d,m-1) = q \sum_{j=1}^{m-1} \beta_j q^{m-1-j} = \sum_{j=1}^m \phi(\beta)_j q^{m-j} = H_s(d,m) \le H_r(d,m). 
\end{equation}
This proves (i). Next, observe that the image of $\phi$ misses the $m$th element of $\Sigma(d,m)$, namely, $(d-1,0, \dots , 0, 1)$. It follows that if $r-1\ge m$ and if $\gamma$ is the $(r-1)$th element of $\Sigma(d,m-1)$, then $\phi(\gamma)$ is the $s$th element of $\Sigma(d,m)$ for some $s \ge r = (r-1)+1$. Thus as in \eqref{qHr}, we see that $qH_{r-1}(d,m-1) \le H_r(d,m)$ whenever $m+1 \le r \le  \binom{m+d-1}{d}$. This proves (ii).
\end{proof}


\subsection{Projective Varieties containing a Hyperplane and Zanella's Theorem for Quadrics}
The following result about projective varieties containing a hyperplane  is a slightly more general version of \cite[Lem. 2.5]{DG}. We include a quick proof for the sake of completeness.
\begin{lemma}\label{lem:factorL}
Assume that $d\le q$.
Let $F_1,\dots,F_r$ be linearly independent homogeneous polynomials in $S_d$. 
Suppose that $L \in S_1$ divides each of $F_1,\dots,F_r$. Then
$$
|\V(F_1,\dots,F_r)| \le H_r(d-1,m)+p_{m-1}.
$$
\end{lemma}
\begin{proof}
The conditions on $L$ show that $L$ is nonzero and thus
without loss of generality, we may assume that $L=X_0$. For $1 \le i \le m$, let  $f_i(X_1,\dots,X_m):= F_i(1,X_1,\dots,X_m)$; note that $\deg(f_i) \le d-1$, since $X_0 \mid F_i$. Hence \eqref{Hrdm} implies that $|\Ze(f_1,\dots,f_r)| \le H_r(d-1,m)$, and so 
$$
|\V(F_1,\dots,F_r)| = |\Ze(f_1,\dots,f_r)|+|\V(X_0)| \le H_r(d-1,m)+p_{m-1},
$$
as desired.
\end{proof}

Note that for the hypothesis of Lemma \ref{lem:factorL} to hold, it is necessary that $r \le \binom{m+d-1}{d-1}$, because otherwise the polynomials $F_1,\dots,F_r$ cannot be linearly independent. Indeed, by assumption, the polynomials $F_1,\dots,F_r$ are in the vector space $L\cdot S_{d-1}$, which has dimension $\binom{m+d-1}{d-1}.$

The last preliminary result we need is the following theorem of Zanella \cite[Thm. 3.4]{Z} about maximum possible number of $\Fq$-rational points on intersections of $r$ linearly independent quadrics in $\PP^m$.

\begin{theorem}\label{thm:Z}
Assume that $r \le \binom{m+2}{2}$.
Let $k$  be the unique integer such that $-1 \le k < m$ and $\binom{m+2}{2}-\binom{k+3}{2}<r \le \binom{m+2}{2}-\binom{k+2}{2}$. Then
$$
e_r(2,m)=\lfloor q^{\binom{m+2}{2}-\binom{k+2}{2}-r-1}\rfloor+p_k .
$$
In particular, if $r \le m+1$, then $k=m-1$ and thus $e_r(2,m)=\lfloor q^{m-r}\rfloor+p_{m-1}$.
\end{theorem}

We have now gathered all known results from the literature that we need. We finish this section by restating the following conjecture from \cite{DG}, which was alluded to in the Introduction.

\begin{conjecture}\label{conj}
Assume that $1 < d < q$ and $1 \le r \le \binom{m+d-1}{d-1}$. Then
$$
e_r(d,m) = H_r(d-1,m)+p_{m-1}.
$$
\end{conjecture}

This conjecture was proved to be correct for $r\le m+1$ and $d<q-1$ in \cite{DG}.
For $r=1$, the conjecture follows from Theorem~\ref{thm:SerreSorr}, whereas
for $d=2$, it 
follows as a particular case of Theorem~\ref{thm:Z} [in view of \eqref{Hvsmallr}], or alternatively, as a  special case of \cite[Thm. 6.3]{DG}. 
Also when $m=1$, the conjecture is a trivial consequence of \eqref{ermone}.
Based on the above, we may always assume that $m>1$, $r>1$, and $d\ge 3$. We will provide significant more evidence for this conjecture by proving it for any pair $(d,r)$ satisfying $2 < d < q$ and $r \le \binom{m+2}{2}$. In particular, we show that the conjecture holds if $d=3$. The main step in our proof would be to show if $r \le \binom{m+2}{2}$ and if $F_1, \dots , F_r$ are any linearly independent 
polynomials in $S_d$, then
\begin{equation}
\label{maineq}
|\V(F_1, \dots , F_r) | \le H_r(d-1,m)+p_{m-1}.
\end{equation}
The equality in Conjecture~\ref{conj} is established by using \eqref{Hrdm} to show that there exists a family of polynomials where the upper the bound in \eqref{maineq} is attained.

\section{Reduction to the relatively prime case}

In order to prove \eqref{maineq} for any linearly independent  $F_1, \dots , F_r\in S_d$, we will establish
in this section auxiliary results that deal with
the case when $\gcd(F_1, \dots , F_r)$ has degree $c>1$. Since
\eqref{maineq} is known already when $r=1$, we will usually assume that $r>1$. Note that when $r>1$, the linear independence of $F_1, \dots , F_r$ implies that $c<d$.

\begin{lemma}\label{lem:factorG}
Assume that $r>1$ and $1< d\le q$. Let $F_1,\dots,F_r\in S_d$ be linearly independent and $G$ be a gcd of $F_1, \dots , F_r$ and let $c:=\deg G$. Let $F'_1, \dots , F'_r\in S_{d-c}$ be such that $F_i = GF'_i$ for $i=1, \dots , r$. Suppose $c >0$ and $G$ has no linear factors. Then
$$
|\V(F_1,\dots,F_r)| < c q^{m-1}+|\V(F_1',\dots,F_r')|.
$$
\end{lemma}
\begin{proof}
Since $r>1$, we must have $c<d$ and so $c\le q-1$. Hence
using Theorem \ref{thm:HK}, we obtain 
$$
|\V(G)| \le (c-1)q^{m-1}+cq^{m-2}+p_{m-3} < c q^{m-1}.
$$
Since we clearly have $|\V(F_1,\dots,F_r)| \le |\V(G)|+|\V(F_1',\dots,F_r')|$, the lemma follows.
\end{proof}
Similar to the remark after Lemma \ref{lem:factorL}, one can deduce that if $G$ and $c$ are as in Lemma \ref{lem:factorG}, then 
we necessarily have
$r \le \binom{m+d-c}{d-c}$, 
since $F_1,\dots,F_r \in G\cdot S_{d-c}$. This gives an alternative argument to show that if $r>1$, then
$c<d$.

\begin{proposition}
\label{prop:factor1}
Assume that $1 < r \le \binom{m+2}{2}$ and $2< d\le q$. 
Let $F_1,\dots,F_r\in S_d$ be linearly independent and $G$ be a gcd of $F_1, \dots , F_r$ and let $c:=\deg G$. Let $F'_1, \dots , F'_r\in S_{d-c}$ be such that $F_i = GF'_i$ for $i=1, \dots , r$.
If $0 < c < d-2$ and $|\V(F_1',\dots,F_r')| \le H_r(d-c-1,m)+p_{m-1}$, then
$$|\V(F_1,\dots,F_r)| < H_r(d-1,m)+p_{m-1}.$$
\end{proposition}
\begin{proof}
If $G$ contains a linear factor and in particular, if $c=1$, then the result follows from Lemma \ref{lem:factorL}. Now suppose $G$ has no linear factors,  $1 < c < d-2$, and 
$|\V(F_1',\dots,F_r')| \le H_r(d-c-1,m)+p_{m-1}$. By
Lemma \ref{lem:factorG}, 
we see that
$$
|\V(F_1,\dots,F_r)| < cq^{m-1}+H_r(d-c-1,m) +p_{m-1}.
$$
On the other hand, changing $d$ to $d-1$ 
in 
\eqref{Hrminust},  we find $H_r(d-1,m)=cq^{m-1}+H_r(d-c-1,m)$. 
This yields the desired inequality. 
%
\end{proof}

The cases $c=d-2$ and $c=d-1$ that are not covered by Proposition \ref{prop:factor1} need to be dealt with independently. However, since the values of $e_r(1,m)$ and $e_r(2,m)$ are known for all 
permissible values of $r$, 
 this is not hard to do.

\begin{proposition}\label{prop:factor2}
Assume that $1 < r \le \binom{m+2}{2}$ and $2< d\le q$.
Let $F_1,\dots,F_r\in S_d$ be linearly independent and let $G$ be a gcd of $F_1, \dots , F_r$. Suppose
$c:=\deg G$ 
equals $d-2$ or $d-1$. Then
$$
|\V(F_1,\dots,F_r)| \le H_r(d-1,m)+p_{m-1}.
$$
\end{proposition}
\begin{proof}
If $G$ contains a linear factor, then Lemma \ref{lem:factorL} gives the desired result. 
Now assume that $G$ has no linear factors. Then Theorem \ref{thm:HK} implies that
$$
|\V(G)| \le (c-1)q^{m-1}+cq^{m-2}+p_{m-3}.
$$
As in the previous proposition, let  $F'_1, \dots , F'_r\in S_{d-c}$ be such that $F_i = GF'_i$ for $i=1, \dots , r$. Then
$|\V(F_1,\dots,F_r)| \le |\V(G)|+|\V(F_1',\dots,F_r')|$. Consequently,
\begin{equation}
\label{eq:estimate}
|\V(F_1,\dots,F_r)| \le (c-1)q^{m-1}+cq^{m-2}+p_{m-3}+ e_r(d-c,m).
\end{equation}
First, let us suppose $c=d-1$. Then we necessarily have $1<r\le m+1$. Also in view of \eqref{ermone},
$e_r(1,m)=p_{m-r} \le p_{m-2}$. Thus \eqref{eq:estimate} implies that 
$$
|\V(F_1,\dots,F_r)| \le 
(d-2)q^{m-1}+(d-1)q^{m-2}+p_{m-3}+ p_{m-2}.
$$
Since $d \le q$, we see that the expression on the right-hand side of the above inequality is
strictly smaller  than 
$(d-2)q^{m-1}+p_{m-1}$.
Hence 
in view of \eqref{Hvsmallr} and \eqref{Monotone}, we see that 
$$
|\V(F_1,\dots,F_r)| < (d-2)q^{m-1}+p_{m-1} = H_{m+1}(d-1,m)+p_{m-1} \le H_r(d-1,m)+p_{m-1},
$$
as desired. 
Next, let us suppose $c=d-2$. Then by Theorem \ref{thm:Z}, we see that
$$
e_r(2,m)=\lfloor q^{m-r} \rfloor+p_{m-1} \le q^{m-2}+p_{m-1} \quad \text{if}  \; 1< r \le m+1,
$$
whereas
$$
e_r(2,m) \le e_{m+2}(2,m)=\lfloor q^{m-2} \rfloor +p_{m-2}  \quad \text{if } \; m+1<r \le \binom{m+2}{2}.
$$
Using this together with \eqref{eq:estimate} and the assumption that $d\le q$, we see that for $1< r \le m+1$,
$$
|\V(F_1,\dots,F_r)| \le (d-3)q^{m-1}+(d-2)q^{m-2}+p_{m-3}+q^{m-2}+p_{m-1} < (d-2)q^{m-1}+p_{m-1},
$$
and thus  in view of \eqref{Monotone}, we find $|\V(F_1,\dots,F_r)| \le H_{m+1}(d-1,m)+p_{m-1} \le H_r(d-1,m)+p_{m-1}$. 
Likewise, when  $m+1<r \le \binom{m+2}{2}$, from  \eqref{eq:estimate} and the assumption $d\le q$ we obtain
%
$$
|\V(F_1,\dots,F_r)| \le (d-3)q^{m-1}+(d-2)q^{m-2}+p_{m-3}+\lfloor q^{m-2} \rfloor +p_{m-2} < (d-3)q^{m-1}+p_{m-1}.
$$
In view of \eqref{Hsmallr},
the expression on the right is $H_{\binom{m+2}{2}}(d-1,m)+p_{m-1}$, which, thanks to \eqref{Monotone},
is less than or equal to $H_r(d-1,m)+p_{m-1}$.
This completes the proof.
\end{proof}

\section{The relatively prime case}
In this section, we will establish results that help in proving \eqref{maineq} when the polynomials $F_1, \dots , F_r$ are relatively prime. Note that for any linearly independent $F_1, \dots , F_r\in S_d$, the corresponding projective variety $\V(F_1, \dots , F_r)$ coincides with $\V(W)$, where $W$ is the $\Fq$-linear subspace of $S_d$ spanned by $F_1, \dots , F_r$. Moreover, we can replace $F_1, \dots , F_r$ by any other basis of $W$. We will thus focus on estimating $|\V(W)|$, where $W$ is any $r$-dimensional subspace of $S_d$ and take $F_1, \dots , F_r$ to be judiciously chosen basis elements of $W$. To this end, an important role will be played by an integer, that we call the \emph{$t$-invariant} of the subspace $W$, which is essentially the largest dimension of the space of polynomials in $W$ that are divisible by a linear homogeneous polynomial. More precisely, given any subspace $W \subseteq S_d$ and $0\ne L \in S_1$, we define $t_W(L):=\dim (W \cap L S_{d-1})$. Note that $0\le t_W(L)\le \dim W$. The \emph{$t$-invariant} of $W$ is defined by 
$$
t_W:=\max\{ t_W(L): L\in S_1, \; L\ne 0\}.
$$
Clearly, $0\le t_W\le \dim W$. Moreover, if $t_W=\dim W=r$, then there exists $0\ne L\in S_1$ such that $L$ divides every element of $W$. In particular, if $W$ is spanned by linearly independent $F_1, \dots , F_r\in S_d$ that are relatively prime, then $t_W < r$. Conversely, if $t_W < r = \dim W$, then for any $F_1, \dots , F_r \in S_d$ that form a basis of $W$, the polynomials $F_1, \dots , F_r$ do not have a common linear factor, or in other words, 
$\V(F_1, \dots , F_r)$ does not contain a hyperplane.

\subsection*{Context}
In this section, we will always assume that $2< d < q$ and $m>1$. Assumption on $r$ may vary and will be specified.

\smallskip

Our first lemma gives a basic set of inequalities that hold under the hypothesis that the inequality \eqref{maineq}, which we wish to prove, holds when $m$ is replaced by $m-1$.

\begin{lemma}
\label{lem:basic}
Assume that 
$1< r \le \binom{m+d-1}{d-1}$. Suppose
\begin{equation}
\label{eq:indhypo}
e_s(d,m-1) \le H_s(d-1,m-1) + p_{m-2} \quad \text{for}\quad 
1 \le s < r. 
\end{equation}
Then for any $r$-dimensional subspace $W$ of $S_d$ with $t:=t_W$ satisfying $1\le t < r$, 
we have
\begin{equation}
\label{eq:estimate2}
|\V(W)| \le H_{r-t}(d-1,m-1) + p_{m-2} + H_t(d-1,m) .
\end{equation}
Moreover, if $t=1$, then
\begin{equation}
\label{eq:estimate3}
|\V(W)| \le H_{r-1}(d-1,m-1) + p_{m-2} + d(d-1) q^{m-2},
\end{equation}
whereas if $t \ge 2$, then
\begin{equation}
\label{eq:estimate4}
|\V(W)| \le H_{r-t}(d-1,m-1) + p_{m-2} + (d-1)^2 q^{m-2}.
\end{equation}
\end{lemma}

\begin{proof}
Let $W$ be any $r$-dimensional subspace $W$ of $S_d$ with $t:=t_W < r$. By a linear change of coordinates, we can and will assume that  $t=t_W(X_0)$. Now we can choose a basis $\{F_1, \dots , F_r\}$ of $W$ such that $\{F_1, \dots , F_t\}$ is a basis of $W\cap X_0S_{d-1}$. Let $F'_1, \dots , F'_t\in S_{d-1}$ be such that $F_i = X_0F'_i$ for $i=1, \dots , t$. Also let $f_1, \dots , f_r$ denote, respectively, the polynomials in $T$ obtained by putting $X_0=1$ in $F_1, \dots , F_r$. Note that $\deg f_i \le d-1$ for $i=1, \dots , t$ and $\deg f_i \le d$ for $i=t+1, \dots , r$. Intersecting $\V(W)$ with the hyperplane $\V(X_0)$ and its complement, we obtain
$$
|\V(W)| = |\V(F_{t+1}, \dots , F_r)\cap \V(X_0)| + |\Ze(f_1, \dots , f_r)| \le e_{r-t}(d,m-1) + |\Ze (f_1, \dots , f_r)| .
$$
Consequently, \eqref{eq:estimate2} follows from \eqref{eq:indhypo} and \eqref{Hrdm}, 
since $|\Ze(f_1, \dots , f_r)| \le |\Ze(f_1, \dots , f_t)|$. 
Moreover, \eqref{eq:estimate3} and \eqref{eq:estimate4} are easily
deduced from the inequality displayed above and Lemma \ref{lem:AffineLachaudrefined}.
\end{proof}

We shall now proceed to refine the inequalities in \eqref{eq:estimate2}--\eqref{eq:estimate4} into \eqref{maineq} by considering separately various possibilities for the $t$-invariant of a given subspace of $S_d$. It will be seen that in many cases we obtain a strict inequality.

\begin{lemma}
\label{lem1}
Assume that 
$1< r \le m+1$. Also suppose \eqref{eq:indhypo} holds. Let $W$ be an $r$-dimensional subspace of  $S_d$ satisfying $t_W = 1$. Then $|\V(W)| < H_r(d-1,m) + p_{m-1}$.
\end{lemma}

\begin{proof}
By Lemma \ref{lem:basic}, we see that \eqref{eq:estimate3} holds. This together with \eqref{Hvsmallr} gives
\begin{eqnarray*}
|\V(W)| &\le & H_{r-1}(d-1,m-1) + p_{m-2} + d(d-1) q^{m-2} \\
& = & (d-2) q^{m-2} +  \lfloor q^{m-r} \rfloor + p_{m-2} + d(d-1) q^{m-2} \\
& \le & (d-2) q^{m-2} +  \lfloor q^{m-r} \rfloor + p_{m-2} + (q-1)(d-2) q^{m-2}+ (q-1) q^{m-2} \\
&=&  (d-2) q^{m-1} +  \lfloor q^{m-r} \rfloor + p_{m-1}  - q^{m-2} \\
&< & H_r(d-1,m) + p_{m-1} ,
\end{eqnarray*}
where the last inequality uses  \eqref{Hvsmallr} and the assumption that $m\ge 2$.
\end{proof}

\begin{lemma}
\label{lem2}
Assume that 
$ 1< r \le {m+1}$. Also suppose \eqref{eq:indhypo} holds. Let $W$ be any $r$-dimensional subspace of  $S_d$ satisfying $2\le t_W <r$. Then $|\V(W)| < H_r(d-1,m) + p_{m-1}$.
\end{lemma}

\begin{proof}
Let $t:=t_W$.
By Lemma \ref{lem:basic}, we see that 
$ |\V(W)| \le  H_{r-t}(d-1,m-1) + p_{m-2} + H_t(d-1,m)$.
Now since $t\le m$ and $r-t\le m-1$, we see from \eqref{Hvsmallr} that
$$
|\V(W)| \le  (d-2) q^{m-2} + q^{(m-1)-(r-t)} + p_{m-2} 
+  (d-2) q^{m-1} + q^{m-t}. 
$$
Further, since $2\le t<r$, we find $(m-1)-(r-t) \le m-2$ and $m-t\le m-2$. Consequently,
$ q^{(m-1)-(r-t)} + q^{m-t} \le 2q^{m-2}$. Thus the above estimate simplifies to
$$
|\V(W)| \le d q^{m-2} + (d-2) q^{m-1} + p_{m-2} \le (d-2) q^{m-1} + p_{m-1}  - q^{m-2} < (d-2) q^{m-1} + p_{m-1} ,
$$
where the second inequality uses 
$d\le q-1$. Also 
$H_r(d-1, m) = (d-2) q^{m-1} + \lfloor q^{m-r} \rfloor$, thanks to \eqref{Hvsmallr}. Thus  $(d-2) q^{m-1} \le H_r(d-1, m)$, 
which yields $|\V(W)| < H_r(d-1,m) + p_{m-1}$.
\end{proof}

\begin{lemma}
\label{lem3}
Assume that 
$m+1< r \le \binom{m+2}{2}$. Also suppose \eqref{eq:indhypo} holds. Let $W$ be any $r$-dimensional subspace of  $S_d$ satisfying $2\le t_W \le m+ 1$. Then $|\V(W)| \le  H_r(d-1,m) + p_{m-1}$.
\end{lemma}

\begin{proof}
Let $t:=t_W$.
By Lemma \ref{lem:basic}, we see that \eqref{eq:estimate4} holds. In view of  \eqref{Monotone}, this gives
$$
|\V(W)| \le  H_{r-t}(d-1,m-1) + p_{m-2} + (d-1)^2 q^{m-2}
\le  H_{r-(m+1)}(d-1,m-1) + p_{m-2} + (d-1)^2 q^{m-2} .
$$
Now since $m+1< r \le \binom{m+2}{2}$, there are unique $i,j\in \Z$ satisfying conditions as in \eqref{rij}, namely,
$$
r = (i-1)m - \binom{i-1}{2} + j \quad \text{and} \quad 2 \le i \le j \le m+1.
$$
This implies that an equation for $r-(m+1)$ such as \eqref{rij}  with $m$ changed to $m-1$,  is given by
$$
r - (m+1) = (i-2)(m-1) - \binom{i-2}{2} + (j-1) \quad \text{and} \quad 1 \le (i-1) \le ( j-1) \le m.
$$
Thus using \eqref{Hsmallr}, we see that $H_r(d-1,m) = (d-3)q^{m-1} +  \lfloor q^{m-i} \rfloor +  \lfloor q^{m-j} \rfloor $ and moreover,
$$
H_{r-(m+1)} (d-1,m-1) = (d-3)q^{m-2} +  \lfloor q^{m-i} \rfloor +  \lfloor q^{m-j} \rfloor ,
$$
where we note that $H_{r-(m+1)} (d-1,m-1)$ is well-defined since $r-(m+1) \le  \binom{m+1}{2}\le  \binom{m+d-2}{d-1}$, thanks to our assumptions on $d,m$ and $r$.
Using this in the above estimate for $|\V(W)| $, we obtain
\begin{eqnarray*}
|\V(W)| &\le & (d-3)q^{m-2} + p_{m-2} + (d-1)^2 q^{m-2} +  \lfloor q^{m-i} \rfloor +  \lfloor q^{m-j} \rfloor \\
& = & (d-2)(d+1) q^{m-2} + p_{m-2} + \lfloor q^{m-i} \rfloor +  \lfloor q^{m-j} \rfloor \\
& \le & (d-2) q^{m-1} + p_{m-2} + \lfloor q^{m-i} \rfloor +  \lfloor q^{m-j} \rfloor\\
& = & (d-3) q^{m-1} +   p_{m-1} + \lfloor q^{m-i} \rfloor +  \lfloor q^{m-j} \rfloor\\
&=&  H_r(d-1,m) + p_{m-1}  ,
\end{eqnarray*}
where the second inequality above uses   the assumption that $d\le q-1$.
\end{proof}

\begin{lemma}
\label{lem4}
Assume that 
$m+1 < r \le \binom{m+2}{2}$. Also suppose \eqref{eq:indhypo} holds. Let $W$ be any $r$-dimensional subspace of  $S_d$ satisfying $m+1 < t_W <r$. Then $|\V(W)| \le H_r(d-1,m) + p_{m-1}$.
\end{lemma}

\begin{proof}
Let $t:=t_W$.
By Lemma \ref{lem:basic}, we see that 
$ |\V(W)| \le  H_{r-t}(d-1,m-1) + p_{m-2} + H_t(d-1,m)$. Here $t\ge m+2$ and $r-t\ge 1$.
Hence from \eqref{Monotone},  we see that
$$
|\V(W)| \le H_1(d-1,m-1) + p_{m-2} + H_{m+2}(d-1,m) .
$$
Consequently, using   \eqref{Hvsmallr} and \eqref{Hmplustwo}, we obtain
\begin{eqnarray*}
|\V(W)| &\le &  (d-2) q^{m-2}   + \lfloor q^{m-2} \rfloor + p_{m-2} + (d-3) q^{m-1} + 2  \lfloor q^{m-2} \rfloor \\
&=& (d-3) q^{m-1} + (d+1)  q^{m-2} + p_{m-2} .
\end{eqnarray*}
Since $d\le q-1$, this gives 
$ |\V(W)| \le  (d-3) q^{m-1} + p_{m-1}$,  
and so in view of \eqref{Hsmallr}, we  conclude that 
$ |\V(W)|  \le H_{r}(d-1,m)+p_{m-1} $.
\end{proof}

It remains to prove \eqref{maineq} when $t_W=0$ and also when $t_W=1$ and $m+1< r \le \binom{m+2}{2}$. Here we need a slightly different technique.

\begin{lemma}
\label{lem5}
Assume that 
$ 1< r \le \binom{m+2}{2}$. Also suppose \eqref{eq:indhypo} holds. Let $W$ be any $r$-dimensional subspace of  $S_d$ satisfying either (i) $t_W=0$ or (ii) $t_W=1$ and $m+1 < r \le \binom{m+2}{2}$. Then $|\V(W)|  \le H_r(d-1,m) + p_{m-1}$.
\end{lemma}

\begin{proof}
Let $t:=t_W$.
Given any hyperplane $\Hy$ in $\PP^m$, we have $\Hy = \V(L)$ for some $0\ne L\in S_1$.
Now $t_W(L) := \dim (W \cap L S_{d-1}) \le t$, and hence in view of  \eqref{eq:indhypo} and \eqref{Monotone}, we see that
$$
|\V(W) \cap \Hy| \le e_{r-t_W(L)}(d-1, m-1) \le H_{r-t_W(L)}(d-1, m-1) + p_{m-2} \le H_{r-t}(d-1, m-1) + p_{m-2}.
$$
Since $\Hy$ was an arbitrary hyperplane in $\PP^m$, using Lemma \ref{lem:Zanella}, we obtain
$$
|\V(W)| \le q \left(H_{r-t}(d-1, m-1) + p_{m-2}\right) + 1 = qH_{r-t}(d-1, m-1) + p_{m-1}.
$$
Hence the desired result follows from parts (i) and (ii) of Proposition \ref{pro:Hrdm}.
\end{proof}

\section{Completion of the Proof}

In this section we combine the results of the previous sections to prove one of our main results.

\subsection*{Context}
As before, $d,m,r$ are fixed positive integers. As in Conjecture \ref{conj}, we generally assume that $1< d< q$. But the relevant assumptions are specified in the statement of the results.

\begin{lemma}
\label{lem:maineq}
Assume that 
$2<d<q$ and $1\le r \le \binom{m+2}{2}$.
Then \eqref{maineq} holds, that is,
$$
|\V(F_1, \dots , F_r)|  \le H_r(d-1,m) + p_{m-1} \quad \text{for any linearly independent } F_1, \dots , F_r\in S_d.
$$
\end{lemma}

\begin{proof}
We use induction on $d+m$. Note that since $d\ge 3$ and $m\ge 1$, we have $d+m\ge 4$ and if $d+m=4$, then $d=3$ and $m=1$, in which case \eqref{maineq} clearly holds, thanks to \eqref{ermone}. Now assume that $d+m>4$ and that  \eqref{maineq} holds for smaller values of $d+m$. Since
 \eqref{maineq} follows from  \eqref{ermone} when $m=1$ and from Theorem~\ref{thm:SerreSorr} when $r=1$,
 we shall henceforth assume that $m>1$ and $r>1$. 

 Let $F_1, \dots , F_r$ be any linearly independent polynomials in $S_d$. Two cases are possible.

 \medskip
 \noindent
 {\bf Case 1.} $F_1, \dots , F_r$ are not relatively prime, i.e., they have a nonconstant common factor.

 In this case, the hypothesis of Proposition \ref{prop:factor1} is satisfied, thanks to the induction hypothesis. Thus from Propositions \ref{prop:factor1} and \ref{prop:factor2}, we see that \eqref{maineq} holds.

\medskip
 \noindent
 {\bf Case 2.} $F_1, \dots , F_r$ are relatively prime.

In this case, \eqref{eq:indhypo} is satisfied, thanks to the induction hypothesis. Further if we let $W$ be the subspace of $S_d$ spanned by $F_1, \dots , F_r$ and let $t=t_W$, then we have $0\le t< r$. Hence from
Lemmas~\ref{lem2}, \ref{lem3}, and \ref{lem4}, we see that \eqref{maineq} holds when $t\ge 2$, whereas
from Lemmas~\ref{lem1}, and \ref{lem5}, we see that \eqref{maineq} holds when $t \le 1$.
This completes the proof.
\end{proof}

The reverse inequality is easy to deduce from the Heijnen-Pellikaan Theorem.

\begin{lemma}
\label{lem:othereq}
Assume that $1 < d \le q$ and $1 \le r \le \binom{m+d-1}{d-1}$. Then
$$
e_r(d,m) \ge H_r(d-1,m)+p_{m-1}.
$$
\end{lemma}

\begin{proof}
Note that $1\le d-1 < q$ and hence by \eqref{Hrdm}, there exist  linearly independent $f_1,\dots,f_r$ in $T_{ \le d-1}$ such that $|\Ze(f_1,\dots,f_r)|=H_r(d-1,m).$ For $1 \le i \le r$, let
$F_i' :=X_0^{d-1}f_i(X_1/X_0, \dots, X_m/X_0)$
and let $F_i:=X_0 F_i'$.
It is easily seen that $F_1,\dots,F_r$ are linearly independent elements of $S_d$ and that $e_r(d,m) \ge |\V(F_1,\dots,F_r)|=H_r(d-1,m)+p_{m-1}$.
\end{proof}

\begin{theorem}\label{thm:main}
Assume that $1 < d < q$ and
$1 \le r\le \binom{m+2}{2}$. Then $e_r(d,m)=H_r(d-1,m)+p_{m-1}$.
\end{theorem}

\begin{proof}
If $d=2$, then the desired result follows from Theorem \ref{thm:Z} as noted in the last paragraph of Section \ref{sec:Prelim}.
If $d>2$, then it is easily seen that $ \binom{m+2}{2} \le \binom{m+d-1}{d-1}$, and so in this case
the desired result follows from Lemmas \ref{lem:maineq} and \ref{lem:othereq}.
\end{proof}

\section{The case $d=q$}

It may have been noted that several of the 
lemmas and propositions in previous sections are actually valid for $d=q$. Thus one may wonder if Conjecture \ref{conj} actually holds for $d =q$ as well. We will answer this here by showing that a
straightforward analogue of Conjecture \ref{conj} is not valid for $d=q$, in general. More precisely, we will determine $e_r(q,m)$ for $1\le r \le m+1$ and show that 
$$
e_r(q,m) > H_r(q-1,m) + p_{m-1} = (q-2)q^{m-1} + q^{m-r}+ p_{m-1} \quad \text{when } q> 2 \text{ and } 1< r \le m.
$$
Note that the case $d=q=2$ is already covered by Theorem \ref{thm:Z}, and here $e_r(q,m)$ behaves as in
Conjecture \ref{conj}. Likewise, when $d=q$ and $r=1$, thanks to Theorem \ref{thm:SerreSorr}.

\begin{lemma}\label{lem:dq1}
Assume that $1 \le r \le m$. Then $$e_r(q,m) \ge q^m+p_{m-r-1}.$$
\end{lemma}
\begin{proof}
For $1 \le i \le r$, consider $F_i\in S_q$ defined by $F_i:=X_i^q-X_0^{q-1}X_i$. Clearly, $F_1, \dots , F_r$ are linearly independent. 
Writing $\X=\V(F_1,\dots,F_r)$ and $\Hy=\V(X_0)$, we see that
$$
|\X \cap \Hy|=|\V(X_0,X_1^q,\dots,X_r^q)|=\left|\{(a_0:a_1 :\dots: a_m) \in \PP^m : a_i=0 \text{ for } 0\le i \le r 
\}\right|=p_{m-r-1}
$$
and
$$
|\X \cap \Hy^c|=|\Ze(X_1^q-X_1,\dots,X_m^q-X_m)|=q^m,
$$
where $\Hy^c$ denotes the complement of $\Hy$ in $\PP^m$. Thus $e_r(q,m) \ge |\X| = q^m+p_{m-r-1}.$
\end{proof}

We shall now show that the lower bound in Lemma \ref{lem:dq1} is, in fact, the exact value of $e_r(q,m)$ when $q\ge 3$. The technique used 
will be similar to that used in the proof of Theorem~\ref{thm:main}.

\begin{theorem}\label{thm:erdm3}
Assume that  $q\ge3$ and $1 \le r \le m$. Then
$$
e_r(q,m) = q^m + p_{m-r-1}.
$$
\end{theorem}

\begin{proof}
In view of Lemma \ref{lem:dq1}, it suffices to show that
\begin{equation}
\label{eqdm}
e_r(q,m) \le q^m+p_{m-r-1} \quad \text{for} \quad 1 \le r \le m.
\end{equation}
We will prove this using induction on $m$. If $m=1$, then \eqref{eqdm} is an immediate consequence of \eqref{ermone}. Now assume that $m>1$ and that \eqref{eqdm}  holds for smaller values of $m$.
 Let $F_1,\dots,F_r \in S_d$ be any linearly independent polynomials, spanning a linear space $W$. We write $t=t_W$ and without loss of generality, we may assume that $t_W = t_W(X_0)$ and also that
  $X_0 \mid F_i$ for $1 \le i \le t$. We shall write $\X=\V(F_1,\dots,F_r)$ and $\Hy=\V(X_0)$, and divide the proof into two 
  cases as follows. 

\medskip
\noindent
{\bf Case 1: } $t=0$.

Here, using the induction hypothesis and the definition of $t$, we see that
$$|
\X \cap \Hy'| \le q^{m-1}+p_{(m-1)-r-1} \quad \text{for every hyperplane $\Hy'$ in } \PP^m \ \text{defined over} \ \Fq.
$$
Hence from Lemma \ref{lem:Zanella}, we obtain \eqref{eqdm}. 

\medskip
\noindent
{\bf Case 2: } $1 \le t  \le r$.

In this case using the induction hypothesis, we obtain
$$
|\X \cap \Hy| \le q^{m-1}+p_{(m-1)-(r-t)-1}
$$
and since $t\le r\le m$ and $2\le (q-1) < q$, from \eqref{Hrdm} and \eqref{Hvsmallr}, we obtain
$$
|\X \cap \Hy^c| \le |\Ze(f_1,\dots,f_t)| \le H_t(q-1,m)=(q-2)q^{m-1}+ q^{m-t} ,
$$
where $\Hy^c$ denote the complement of $\Hy = \V(X_0)$ in $\PP^m$.
Therefore we have
$$
|\X|  \le  p_{m-r+t-2}+(q-1)q^{m-1}+ q^{m-t}  = q^m+p_{m-r-1}+R,
$$
where
$$
R = p_{m-r+t-2} -p_{m-r-1} - q^{m-1}+ q^{m-t} = \frac{ (q^{t-1}-1) \left(q^{m-r}-q^{m-t+1} + q^{m-t} \right)}{(q-1)} \le 0,
$$
where the last inequality follows since $m-r \le m-t$ and $q\ge 2$. This proves \eqref{eqdm}.
%
\end{proof}

A special case of Theorem \ref{thm:erdm3} is that if $q>2$, then
$e_1(q,m)=q^m+p_{m-2}$, and from \eqref{Hvsmallr}, we see that this equals $H_1(q-1,m)+p_{m-1}$.  
However, when  $q>2$ and $1<r\le m$, by 
substituting $p_{m-r-1} = (q^{m-r}-1)/(q-1)$ and $p_{m-1} = (q^m-1)/(q-1)$, 
an elementary calculation shows that 
$$
\left(q^m+p_{m-r-1} \right) -  (q-2)q^{m-1} -q^{m-r} - p_{m-1} =
\frac{\left( q - 2 \right) \left( q^{m-1} - q^{m-r} \right)}{(q-1)} > 0,
$$
and so $e_r(q,m) > H_r(q-1,m) + p_{m-1}$. Thus, Conjecture \ref{conj} does not hold for $d=q$ in general. 
Perhaps somewhat surprisingly, it turns out that Conjecture \ref{conj} is valid when $r=m+1$ and $d=q$. 
The proof follows a very similar pattern as in Theorem \ref{thm:erdm3}

\begin{theorem}
\label{pro:erdm4}
Assume that $q \ge 3$. Then
$$e_{m+1}(q,m)=(q-1)q^{m-1}+p_{m-2}.$$
\end{theorem}

\begin{proof}
We will show using induction on $m$ that $e_{m+1}(q,m) \le (q-1)q^{m-1}+p_{m-2}.$ When $m=1$, this follows from \eqref{ermone}. 
Assume that $m>1$ and that the inequality holds for smaller values of $m$. Let $F_1,\dots,F_{m+1}$ be any linearly independent
polynomials in $S_q$, and let $W$ be the linear space spanned by them. Write $t=t_W$ and assume without loss of generality that $t=t_W(X_0)$ and also that
  $X_0 \mid F_i$ for $1 \le i \le t$. We shall write $\X=\V(F_1,\dots,F_r)$ and $\Hy=\V(X_0)$, and divide the proof into two cases as follows.

\medskip
\noindent
{\bf Case 1: } $t=0$ or $t=1$.

Using the induction hypothesis, we obtain for any hyperplane $\Hy'$ in $\PP^m$ defined over $\Fq$,
$$
|\X \cap \Hy'| \le |\V(F_2,\dots,F_{m+1}) \cap \Hy'| \le (q-1)q^{m-2}+p_{(m-1)-2}.
$$
Hence using Lemma \ref{lem:Zanella} we obtain $|\X| \le  (q-1)q^{m-1}+p_{m-2}.$

\medskip
\noindent
{\bf Case 2: } $2 \le t \le m+1$.

Here, we can apply Theorem \ref{thm:erdm3}  and it gives
$$
|\X \cap \Hy| \le q^{m-1}+ p_{(m-1)-(m+1-t)-1}=q^{m-1}+p_{t-3}.
$$
Moreover, using \eqref{Hrdm} and \eqref{Hvsmallr}, we obtain
$$
|\X \cap \Hy^c| \le |\Ze(f_1,\dots,f_{t})| \le H_t(q-1,m)=(q-2)q^{m-1}+\lfloor q^{m-t} \rfloor,
$$
where $f_i\in T_{\le m}$ is obtained by putting $X_0=1$ in $F_i$ for $1\le i \le t$ and $\Hy^c = \PP^m \setminus \Hy$.
Hence
$$
|\X|   \le  q^{m-1}+p_{t-3}+(q-2)q^{m-1}+\lfloor q^{m-t} \rfloor
= (q-1)q^{m-1}+p_{t-3}+\lfloor q^{m-t}\rfloor.
$$
Since $2\le t\le m+1$, this implies $|\X| \le  (q-1)q^{m-1}+p_{m-2}$.

It follows  that $e_{m+1}(q,m) \le (q-1)q^{m-1}+p_{m-2}$.
The reverse inequality follows from
Lemma~\ref{lem:othereq}. This completes the proof.
\end{proof}

It is thus seen that the formulas for $e_r(q,m)$ obtained in this section for $1\le r \le m+1$ are of a different kind than those for $e_r(d,m)$ when $d<q$.
The general pattern for $e_r(q,m)$ for $1\le r \le \binom{m+q}{q}$ does not seem clear, even conjecturally. At any rate, it remains an interesting open problem to determine $e_r(d,m)$ for all the remaining values of $r$ and $m$ when $1<d < q$ and also when $d=q$.

\section{Acknowledgements}

The authors would like to gratefully acknowledge the following foundations and institutions:
Peter Beelen is supported by The Danish Council for Independent Research (Grant No. DFF--4002-00367). Mrinmoy Datta is supported by The Danish Council for Independent Research (Grant No. DFF--6108-00362). Sudhir Ghorpade is partially supported by IRCC Award grant 12IRAWD009
from IIT Bombay. 
Also, Peter Beelen would like to thank IIT Bombay where large parts of this work were carried out when he was there in January 2016 as a Visiting Professor. 
Sudhir Ghorpade would like to thank the Technical University of Denmark for a visit of 10 days in June-July 2016 during which this work was completed.  We are also grateful to an anonymous referee for many useful comments and suggestions.

\end{document}